\documentclass{amsart}%
\usepackage{amssymb}
\usepackage{amsfonts}
\usepackage{amsmath}
\usepackage{graphicx}%
\providecommand{\U}[1]{\protect\rule{.1in}{.1in}}
\newtheorem{theorem}{Theorem}
\theoremstyle{plain}

\newtheorem{corollary}{Corollary}

\newtheorem{lemma}{Lemma}

\newtheorem{proposition}{Proposition}
\newtheorem{remark}{Remark}

\numberwithin{equation}{section}
\begin{document}
\title{Function Spaces Related to the Dirichlet Space }
\author{N. Arcozzi}
\address{Dipartimento do Matematica\\
Universita di Bologna\\
40127 Bologna, ITALY}
\author{R. Rochberg}
\address{Department of Mathematics\\
Washington University\\
St. Louis, MO 63130, U.S.A}
\author{E. Sawyer}
\address{Department of Mathematics \& Statistics\\
McMaster University\\
Hamilton, Ontairo, L8S 4K1, CANADA}
\author{B. D. Wick}
\address{Department of Mathematics\\
LeConte College\\
1523 Greene Street\\
University of South Carolina\\
Columbia, SC 29208}
\thanks{N.A.'s work partially supported by the COFIN project Analisi Armonica, funded
by the Italian Minister for Research}
\thanks{R.R.'s work supported by the National Science Foundation under Grant No. 0070642}
\thanks{E.S.'s work supported by the National Science and Engineering Council of Canada.}
\thanks{B.W.'s work supported by the National Science Foundation under Grant No. 0752703}
\maketitle

\section{Introduction}

We present results about spaces holomorphic of functions associated to the
classical Dirichlet space. The spaces we consider have roles similar to the
roles of $H^{1}$ and $BMO$ in the Hardy space theory and we will emphasize
those analogies.

Definitions and background information are in the next section. The sections
after that contain results about the spaces and the functions in them. Most of
the results are new but some are not; however it seemed useful to present all
of them together. In the brief final section we mention some questions.

\section{Background}

We begin by defining the Dirichlet and Hardy spaces in ways that emphasize the
analogy between them. General background references for the Hardy space theory
include \cite{G}, \cite{D}, and \cite{N}; further information about the
Dirichlet space is in \cite{Ro} and \cite{W2}.

The Dirichlet space\ $\mathcal{D}$ is the Hilbert space of holomorphic
functions $f=\sum_{n=0}^{\infty}a_{n}z^{n}$ on the unit disk $\mathbb{D}$ for
which
\[
\int_{\mathbb{D}}\left\vert f^{\prime}\left(  z\right)  \right\vert
^{2}dA<\infty,\text{ equivalently, }\left\Vert f\right\Vert _{\mathcal{D}}%
^{2}=\sum_{0}^{\infty}\left(  n+1\right)  \left\vert a_{n}\right\vert
^{2}<\infty.
\]
The $\mathcal{D}$ inner product of $f$ and $g=\sum_{n=0}^{\infty}b_{n}z^{n}$
is given by%
\begin{align*}
\left\langle f,g\right\rangle  &  =\left\langle f,g\right\rangle
_{\mathcal{D}}\\
&  =\sum_{0}^{\infty}\left(  n+1\right)  a_{n}\bar{b}_{n}\\
&  \sim f(0)\overline{g(0)}+\int_{\mathbb{D}}f^{\prime}(z)\overline{g^{\prime
}(z)}dA.
\end{align*}
The Hardy space $H^{2}$ is the Hilbert space of holomorphic functions on the
unit disk for which
\[
\int_{\mathbb{D}}\left\vert f^{\prime}\left(  z\right)  \right\vert
^{2}(1-\left\vert z\right\vert ^{2})dA<\infty,\text{ equivalently, }\left\Vert
f\right\Vert _{H^{2}}^{2}=\sum_{0}^{\infty}\left\vert a_{n}\right\vert
^{2}<\infty.
\]
There are fundamental differences between the functional analytic and function
theoretic results for these spaces but there are also intriguing analogies,
some of which we will see below.

Associated to the Hankel bilinear forms which we will consider in Section
\ref{matrices} are "weakly factored" function spaces which we now define; see
\cite{A}, \cite{ARSW}, \cite{CFR}, and \cite{CV} for other instances of this
construction. Define the weakly factored space $\mathcal{D}\odot\mathcal{D}$
to be the completion of finite sums $h=\sum f_{j}g_{j}$ using the norm
\[
\left\Vert h\right\Vert _{\mathcal{D}\odot\mathcal{D}}=\inf\left\{
\sum\left\Vert f_{j}\right\Vert _{\mathcal{D}}\left\Vert g_{j}\right\Vert
_{\mathcal{D}}:h=\sum f_{j}g_{j}\right\}  .
\]
In particular if $f\in\mathcal{D}$ then $f^{2}\in\mathcal{D}\odot\mathcal{D}$
and
\begin{equation}
\left\Vert f^{2}\right\Vert _{\mathcal{D}\odot\mathcal{D}}\leq\left\Vert
f\right\Vert _{\mathcal{D}}^{2}. \label{square}%
\end{equation}
The spaces $H^{2}\odot H^{2}$ is defined analogously using the norm of $H^{2}
$. It is an immediate consequence of the inner-outer factorization for Hardy
space functions that $H^{2}\odot H^{2}=H^{1}$.

We also introduce a variant of $\mathcal{D}\odot\mathcal{D}$. Define the space
$\partial^{-1}\left(  \partial\mathcal{D}\odot\mathcal{D}\right)  $ to be the
completion of the space of functions $h$ such that $h^{\prime}$ can be written
as a finite sum, $h^{\prime}=\sum f_{j}^{\prime}g_{j}$ (and thus
$h=\partial^{-1}\sum\left(  \partial f_{i}\right)  g_{i})$, with the norm
\[
\left\Vert h\right\Vert _{\partial^{-1}\left(  \partial\mathcal{D}%
\odot\mathcal{D}\right)  }=\inf\left\{  \sum\left\Vert f_{j}\right\Vert
_{\mathcal{D}}\left\Vert g_{j}\right\Vert _{\mathcal{D}}:h^{\prime}=\sum
f_{j}^{\prime}g_{j}\right\}  .
\]

We say that a positive measure $\mu$ supported on the closed disk is a
Carleson measure for $\mathcal{D}$, $\mu\in CM(\mathcal{D)}$, if there is a
$C>0$ so that for all $f\in$ $\mathcal{D}$
\[
\int_{\mathbb{D}}\left\vert f\right\vert ^{2}d\mu\leq C^{2}\left\Vert
f\right\Vert _{\mathcal{D}}^{2}.
\]
The smallest such $C$ is the Carleson measure norm of $\mu,$ $\left\Vert
\mu\right\Vert _{CM(\mathcal{D)}}$. The set of measures $CM(H^{2}\mathcal{)} $
is defined and normed analogously. The measures in $CM(\mathcal{D)}$ were
first characterized by Stegenga \cite{S} using capacity theoretic conditions.
Measure theoretic characterizations can also be given, for instance in
\cite{ARS2}.

Recall that among the equivalent definitions of $BMO$ is that $f$ is in $BMO$
exactly if%
\begin{equation}
\left\Vert f\right\Vert _{BMO}^{2}\sim\left\vert f(0)\right\vert
^{2}+\left\Vert \left\vert f^{\prime}\right\vert ^{2}(1-\left\vert
z\right\vert ^{2})dA\right\Vert _{CM(H^{2}\mathcal{)}}<\infty. \label{bmo}%
\end{equation}
We next introduce the space $\mathcal{X}$ which plays a role in the Dirichlet
space theory analogous to the role of $BMO$ in the Hardy space theory. We say
$f\in\mathcal{X}$ if
\[
\left\Vert f\right\Vert _{\mathcal{X}}^{2}=\left\vert f(0)\right\vert
^{2}+\left\Vert \left\vert f^{\prime}\right\vert ^{2}dA\right\Vert
_{CM(\mathcal{D)}}<\infty.
\]
We denote the closure in $\mathcal{X}$ of the space of polynomials by
$\mathcal{X}_{0}.$

Finally we define the multiplier spaces. For a space of holomorphic functions
$X$ the multiplier space, $\mathcal{M}(X),$ is the space of functions $f$ for
which multiplication by $f$ is a bounded map of $X$ into itself. The space is
normed by the norm of the multiplication operator. It is a result going back
to Setgenga \cite{S} that $\mathcal{M}(\mathcal{D})=H^{\infty}\cap\mathcal{X}%
$; here $H^{\infty}$ is the space of bounded holomorphic functions on the
disk. The analogous result $\mathcal{M}(H^{2})=H^{\infty}\cap BMO$ also holds
but it is never presented that way because $H^{\infty}\subset BMO.$

Here is a summary of relations between the spaces. The duality pairings are
with respect to the Dirichlet pairing $\left\langle \cdot,\cdot\right\rangle
_{\mathcal{D}}.$

\begin{theorem}
\label{duality} We have

\begin{enumerate}
\item $\mathcal{X}_{0}^{\ast}=\mathcal{D}\odot\mathcal{D},$

\item $\left(  \mathcal{D}\odot\mathcal{D}\right)  ^{\ast}=\mathcal{X},$

\item $\mathcal{M}(\mathcal{D})=H^{\infty}\cap\mathcal{X},$

\item $\mathcal{D}\odot\mathcal{D=}$ $\partial^{-1}\left(  \partial
\mathcal{D}\odot\mathcal{D}\right)  .$
\end{enumerate}

\begin{proof}
[Proof discussion]As we mentioned (3) is proved in \cite{S}.

A result essentially equivalent to $\left(  \partial^{-1}\left(
\partial\mathcal{D}\odot\mathcal{D}\right)  \right)  ^{\ast}=\mathcal{X}$ was
proved by Coifman-Muri \cite{CM} using real variable techniques and in more
function theoretic contexts by Tolokonnikov \cite{To} and by Rochberg-Wu
\cite{RW}. An interesting alternative approach to the result is given by Treil
and Volberg in \cite{TV}.

In \cite{W1} it is shown that $\mathcal{X}_{0}^{\ast}=\partial^{-1}\left(
\partial\mathcal{D}\odot\mathcal{D}\right)  .$ Item (2) is proved in
\cite{ARSW} and when that is combined with the other results we obtain (1) and (4).
\end{proof}
\end{theorem}

Statement (2) of the theorem is the analog of Nehari's characterization of
bounded Hankel forms on the Hardy space, recast using the identification
$H^{2}\odot H^{2}=H^{1}$ and Fefferman's duality theorem. Item (1) is the
analog of Hartman's characterization of compact Hankel forms. Statement (4) is
similar in spirit to the weak factorization result for Hardy spaces given by
Aleksandrov and Peller in \cite{AP} where they study Foguel-Hankel operators
on the Hardy space.

Given the previous theorem it is easy to check the inclusions%

\begin{equation}
\mathcal{M}(\mathcal{D})\subset\mathcal{X\subset D\subset D}\odot\mathcal{D}
\label{dirinc}%
\end{equation}
which we will use later.

\section{Size}

In this section we obtain norm and pointwise estimates for the functions in
$\mathcal{D}\odot\mathcal{D}$ and $\mathcal{X}$. We begin by recalling the
basic results for the Dirichlet space.

For $\zeta\in\mathbb{D}$ define functions $\delta$ and $L$ by%
\begin{align*}
\delta(\zeta)  &  =1-\left\vert \zeta\right\vert ^{2}\\
L(\zeta)  &  =1+\log\delta(\zeta)^{-1}.
\end{align*}

Recall \cite[Section 1.1]{G} that the pseudohyperbolic metric, $\rho,$ on the
disk is given by
\[
\rho(\zeta_{1},\zeta_{2})=\left\vert \frac{\zeta_{1}-\zeta_{2}}{1-\overline
{\zeta_{1}}\zeta_{2}}\right\vert
\]
and satisfies $0\leq\rho<1.$ The hyperbolic distance, $\beta,$ is given by
\begin{equation}
\beta(\zeta_{1},\zeta_{2})=\log\left(  \frac{1+\rho(\zeta_{1},\zeta_{2}%
)}{1-\rho(\zeta_{1},\zeta_{2})}\right)  . \label{related}%
\end{equation}
On subsets of $\mathbb{D}\times\mathbb{D}$ on which $\rho\leq c<1$ we have
$\rho\sim\beta$. However for highly separated points we have estimates such as
$\beta(0,\zeta)\sim L(\zeta)\ $as $\left\vert \zeta\right\vert \rightarrow1.$

For $\zeta\in\mathbb{D}$ the reproducing kernel (for $\mathcal{D)}$, which is
characterized by the property that for $f\in\mathcal{D}$, $\zeta\in\mathbb{D}$
we have $f(\zeta)=\left\langle f,k_{\zeta}\right\rangle $, is given by the
formula
\[
k_{\zeta}(z)=\frac{1}{\bar{\zeta}z}\log\frac{1}{\left(  1-\bar{\zeta}z\right)
}.
\]
One has that for $\zeta,$ $\zeta_{1},$ $\zeta_{2}\in\mathbb{D}$
\begin{align}
\left\Vert k_{\zeta}\right\Vert _{\mathcal{D}}  &  =L(\zeta)^{1/2},\text{
}\label{x}\\
\left\Vert \bar{\partial}_{\zeta}k_{\zeta}\right\Vert _{\mathcal{D}}  &
\sim\delta(\zeta)^{-1},\label{y}\\
\left\Vert k_{\zeta_{1}}-k_{\zeta_{2}}\right\Vert _{\mathcal{D}}  &  \sim
\frac{\beta(\zeta_{1},\zeta_{2})}{1+\beta(\zeta_{1},\zeta_{2})^{1/2}}%
\lesssim\beta(\zeta_{1},\zeta_{2})^{1/2}. \label{z}%
\end{align}
The first two are straightforward. For the third we introduce the space
$\mathcal{\tilde{D}}$ of functions in $\mathcal{D}$ which vanish at the origin
and which is normed by $\left\Vert \sum_{1}^{\infty}a_{n}z^{n}\right\Vert
_{\mathcal{\tilde{D}}}^{2}=\sum n\left\vert a_{n}\right\vert ^{2}.$ The
reproducing kernels for $\mathcal{\tilde{D}}$ are the functions $\tilde
{k}_{\zeta}(z)=-\log\left(  1-\bar{\zeta}z\right)  . $ We have
\begin{align}
\left\Vert k_{\zeta_{1}}-k_{\zeta_{2}}\right\Vert _{\mathcal{D}}^{2}  &
=\sup\left\{  \left\vert \left\langle f,k_{\zeta_{1}}-k_{\zeta_{2}%
}\right\rangle \right\vert :f\in\mathcal{D},\left\Vert f\right\Vert
_{\mathcal{D}}=1\right\}  ^{2}\nonumber\\
&  =\sup\left\{  \left\vert \left\langle f,k_{\zeta_{1}}-k_{\zeta_{2}%
}\right\rangle \right\vert :f\in\mathcal{D},f(0)=0,\left\Vert f\right\Vert
_{\mathcal{D}}=1\right\}  ^{2}\nonumber\\
&  \sim\sup\left\{  \left\vert \left\langle f,\tilde{k}_{\zeta_{1}}-\tilde
{k}_{\zeta_{2}}\right\rangle _{\mathcal{\tilde{D}}}\right\vert :f\in
\mathcal{\tilde{D}},\left\Vert f\right\Vert _{\mathcal{\tilde{D}}}=1\right\}
^{2}\nonumber\\
&  =\left\Vert \tilde{k}_{\zeta_{1}}-\tilde{k}_{\zeta_{2}}^{2}\right\Vert
_{\mathcal{\tilde{D}}}=\left\langle \tilde{k}_{\zeta_{1}}-\tilde{k}_{\zeta
_{2}},\tilde{k}_{\zeta_{1}}-\tilde{k}_{\zeta_{2}}\right\rangle
_{\mathcal{\tilde{D}}}\nonumber\\
&  =-\log\left(  1-\left\vert \zeta_{1}\right\vert ^{2}\right)  -\log\left(
1-\left\vert \zeta_{2}\right\vert ^{2}\right)  +2\log\left\vert 1-\overline
{\zeta_{1}}\zeta_{2}\right\vert \nonumber\\
&  =-\log\frac{\left(  1-\left\vert \zeta_{1}\right\vert ^{2}\right)  \left(
1-\left\vert \zeta_{2}\right\vert ^{2}\right)  }{\left\vert 1-\overline
{\zeta_{1}}\zeta_{2}\right\vert ^{2}}=-\log\left(  1-\left\vert \frac
{\zeta_{1}-\zeta_{2}}{1-\overline{\zeta_{1}}\zeta_{2}}\right\vert ^{2}\right)
\label{use}\\
&  =\log\left(  \frac{1+\rho(\zeta_{1},\zeta_{2})}{1-\rho(\zeta_{1},\zeta
_{2})}\right)  -2\log\left(  1+\rho(\zeta_{1},\zeta_{2})\right) \label{smll}\\
&  =\beta(\zeta_{1},\zeta_{2})-2\log\left(  1+\rho(\zeta_{1},\zeta
_{2})\right)  . \label{large}%
\end{align}
The passage from the first line to the second uses the fact that replacing
$f(z)$ by $f(z)-f(0)$ gives a better competitor for calculating the first
supremum. The passage to the third line uses the fact that the identity map is
a bounded invertible map of $\left\{  f\in\mathcal{D},f(0)=0\right\}  $ to
$\mathcal{\tilde{D}}$. The equality in (\ref{use}) is a computational
identity, see \cite[Section 1.1]{G}, and the final line is obtained using
(\ref{related}).

If $\rho\leq c<1$ then we have $\rho\sim\beta$ and the desired estimate can be
seen from from (\ref{smll}). For $\rho\sim1$ we estimate using the last line.

From these estimates follow pointwise estimates for $f\in\mathcal{D}$;
\begin{align}
\sup\left\{  \left\vert f(\zeta)\right\vert :\left\Vert f\right\Vert
_{\mathcal{D}}\leq1\right\}   &  \sim L(\zeta)^{1/2},\label{dvalue}\\
\sup\left\{  \left\vert f^{\prime}(\zeta)\right\vert :\left\Vert f\right\Vert
_{\mathcal{D}}\leq1\right\}   &  \sim\delta(z)^{-1},\label{dderiv}\\
\sup\left\{  \left\vert f(\zeta_{1})-f(\zeta_{2})\right\vert :\left\Vert
f\right\Vert _{\mathcal{D}}\leq1\right\}   &  \lesssim\beta(\zeta_{1}%
,\zeta_{2})^{1/2}. \label{ddiff}%
\end{align}
In such estimates will refer to the fact that the left side is dominated by
the right as the \textit{upper estimate,} the other as the \textit{lower
estimate. }

We now give estimates for $k_{\zeta}$, $\bar{\partial}_{\zeta}k_{\zeta}$ and
related functions in $\mathcal{X}\ $and $\mathcal{D}\odot\mathcal{D}.$ We omit
rewriting them in the forms such as (\ref{dvalue}), (\ref{dderiv}), and
(\ref{ddiff}).

\begin{theorem}
For $\zeta,\zeta_{1},\zeta_{2}\in\mathbb{D}$ we have

\begin{enumerate}
\item (norm estimates in $\mathcal{X}$)
\begin{align}
\left\Vert k_{\zeta}\right\Vert _{\mathcal{X}}  &  \sim L(\zeta),\label{xval}%
\\
\left\Vert \bar{\partial}_{\zeta}k_{\zeta}\right\Vert _{\mathcal{X}}  &  \sim
L(\zeta)^{1/2}\delta(\zeta)^{-1},\label{xderiv}\\
\left\Vert k_{\zeta_{1}}-k_{\zeta_{2}}\right\Vert _{\mathcal{X}}  &
\lesssim\beta(\zeta_{1},\zeta_{2})^{1/2}\left(  L(\zeta_{1})^{1/2}+L(\zeta
_{2})^{1/2}\right)  , \label{xdiff}%
\end{align}

\item (norm estimates in $\mathcal{D}\odot\mathcal{D)}$%
\begin{align}
\left\Vert k_{\zeta}\right\Vert _{\mathcal{D}\odot\mathcal{D}}  &  \sim
\log(1+L(\zeta)),\text{ }\label{tval}\\
\left\Vert k_{\zeta}^{2}\right\Vert _{\mathcal{D}\odot\mathcal{D}}  &  \sim
L(\zeta),\label{tsquare}\\
\left\Vert \bar{\partial}_{\zeta}k_{\zeta}\right\Vert _{\mathcal{D}%
\odot\mathcal{D}}  &  \sim L(\zeta)^{-1/2}\delta(\zeta)^{-1}. \label{tderiv}%
\end{align}
For $\theta>0$ we have
\begin{equation}
\left\Vert \left(  \frac{1-\left\vert \zeta\right\vert ^{2}}{1-\bar{\zeta}%
z}\right)  ^{\theta}\right\Vert _{_{\mathcal{D}\odot\mathcal{D}}}\sim
L(\zeta)^{-1/2}. \label{bump}%
\end{equation}
The implied constant here may depend on $\theta.$

\item (a relatively large function in $\mathcal{X)}$

The function $H(z)$ defined by
\[
H(z)=\int_{1/2}^{1}\log\left(  \frac{1}{1-zx}\right)  \frac{1}{\left(
1-x\right)  \left[  \log\left(  1-x\right)  \right]  ^{2}}dx
\]
satisfies both%
\begin{align}
\left\Vert H\right\Vert _{\mathcal{X}}  &  <\infty\text{, and}\label{H}\\
H(\zeta)  &  =\log(\left\vert \log(1-\zeta)\right\vert )+O(1)\text{ for }%
\zeta\in(.9,1).\nonumber
\end{align}

\item (norms of monomials)

For $n=1,2,...,$%
\begin{align}
\left\Vert z^{n}\right\Vert _{\mathcal{D}\odot\mathcal{D}}  &  \sim\sqrt
{n}\text{,}\label{ten}\\
\left\Vert z^{n}\right\Vert _{\mathcal{X}}  &  \sim\sqrt{n}. \label{X}%
\end{align}

\end{enumerate}
\end{theorem}

\begin{proof}
We first note that the estimates (\ref{x}), (\ref{y}) and (\ref{z}) imply the
upper estimates in (\ref{xval}), (\ref{xderiv}), and (\ref{xdiff}). More
precisely if one starts with a representation $h=\sum f_{j}g_{j}$ of
$h\in\mathcal{D}\odot\mathcal{D}$ which is almost optimal and then applies the
estimate (\ref{dvalue}) to all of the $f^{\prime}s$ and $g^{\prime}s$ we find
$\left\vert h(\zeta)\right\vert \lesssim L(\zeta)\left\Vert h\right\Vert
_{\mathcal{D}\odot\mathcal{D}}.$ Taking note of the fact that $\left(
\mathcal{D}\odot\mathcal{D}\right)  ^{\ast}=\mathcal{X}$ this gives the upper
estimate in (\ref{xval}). If we start with the same representation of $h$,
compute $h^{\prime}$ and apply the estimates (\ref{dvalue}) and (\ref{dderiv})
we conclude
\[
\left\vert h^{\prime}(\zeta)\right\vert \lesssim L(\zeta)^{1/2}\delta
(x)^{-1}\left\Vert h\right\Vert _{\mathcal{D}\odot\mathcal{D}}.
\]
This gives the upper estimate\ for a functional in the dual space, this time
the upper estimate in (\ref{xderiv}). Similarly, we obtain (\ref{xdiff}) by
showing that for a unit vector $h\in\mathcal{D}\odot\mathcal{D}$ we have a
good estimate for $\left\vert h(\zeta_{1})-h(\zeta_{2})\right\vert $. Using
the identity
\begin{equation}
fg=\frac{1}{4}((f+g)^{2}-(f-g)^{2}) \label{pol}%
\end{equation}
we see that we can write $h=\sum h_{j}^{2}$ with $\sum\left\Vert
h_{j}\right\Vert _{\mathcal{D}}^{2}=O(1).$ For each term we have
\[
\left\vert h_{j}^{2}(\zeta_{1})-h_{j}^{2}(\zeta_{2})\right\vert \leq\left\vert
h_{j}(\zeta_{1})-h_{j}(\zeta_{2})\right\vert \left\{  \left\vert h_{j}%
(\zeta_{1})|+|h_{j}(\zeta_{2})\right\vert \right\}  .
\]
We apply (\ref{z}) to the first factor on the right and (\ref{x}) to the terms
inside the braces and obtain%
\[
\left\vert h_{j}^{2}(\zeta_{1})-h_{j}^{2}(\zeta_{2})\right\vert \lesssim
\beta(\zeta_{1},\zeta_{2})^{1/2}\left(  L(\zeta_{1}\right)  ^{1/2}+\left(
L(\zeta_{2}\right)  ^{1/2})\left\Vert h_{j}\right\Vert _{\mathcal{D}}^{2}.
\]
Summing with respect to $j$ gives the desired estimate for $h.$

We now consider the corresponding lower estimates. Note that $k_{\zeta}%
(\zeta)\sim L(\zeta).$ Using the upper estimate in (\ref{xval}) and duality we
have
\[
L(\zeta)^{2}\sim\left\vert k_{\zeta}(\zeta)^{2}\right\vert =\left\vert
\left\langle k_{\zeta}^{2},k_{\zeta}\right\rangle \right\vert \leq\left\Vert
k_{\zeta}^{2}\right\Vert _{\mathcal{D}\odot\mathcal{D}}\left\Vert k_{\zeta
}\right\Vert _{\mathcal{X}}\lesssim\left\Vert k_{\zeta}^{2}\right\Vert
_{\mathcal{D}\odot\mathcal{D}}L(\zeta).
\]
Comparing the right side and the left we obtain the lower estimate in
(\ref{tsquare}). With that estimate in hand we compare the left side with the
fourth term and obtain the lower estimate in (\ref{xval}).

We now prove the upper estimate for (\ref{bump}) as a separate lemma.

\begin{lemma}
Pick and fix $\theta>0$ and $\zeta\in\mathcal{D}$. Define%
\[
G_{\zeta,\theta}(z)=G_{\theta}(z)=\left(  \frac{1-\left\vert \zeta\right\vert
^{2}}{1-\bar{\zeta}z}\right)  ^{\theta}.
\]
We have
\[
\left\Vert G_{\zeta,\theta}\right\Vert _{_{\mathcal{D}\odot\mathcal{D}}%
}\lesssim L(\zeta)^{-1/2}.
\]

\begin{proof}
[Proof of Lemma]We will use the auxiliary function $\Lambda$,
\[
\Lambda(z)=3i-\log\left(  1-\bar{\zeta}z\right)  .
\]
The constant $3i$ insures $\operatorname{Im}\left(  \Lambda\right)  >1$ and in
particular we can work freely with powers of $\Lambda.$ Set $G_{1}=G_{\theta
}\Lambda^{-3/4}\ $and $G_{2}=\Lambda^{3/4};$ thus $G_{\theta}=G_{1}G_{2}.$ We
will obtain the upper estimate in (\ref{bump}) using
\begin{equation}
\left\Vert G_{\theta}\right\Vert _{_{\mathcal{D}\odot\mathcal{D}}}^{2}%
\leq\left\Vert G_{1}\right\Vert _{_{\mathcal{D}}}^{2}\left\Vert G_{2}%
\right\Vert _{_{\mathcal{D}}}^{2}. \label{product}%
\end{equation}
Without loss of generality we assume $\zeta$ is real and positive. We only
need to consider the case of $\zeta$ close to $1$. Set $\zeta=1-\gamma$ and
$z=1+w$, hence $\gamma\sim\delta=\delta(\zeta)$ and $1-\bar{\zeta}%
z=\gamma-w+\gamma w.$ We compute%
\[
G_{1}^{\prime}(z)=\theta\bar{\zeta}\frac{\delta^{\theta}}{\left(  1-\bar
{\zeta}z\right)  ^{\theta+1}\Lambda^{3/4}}+\frac{3\bar{\zeta}}{4}\frac
{\delta^{\theta}}{\left(  1-\bar{\zeta}z\right)  ^{\theta+1}\Lambda^{7/4}}%
\]
and%
\[
G_{2}^{\prime}(z)=\frac{3\bar{\zeta}}{4}\frac{1}{\left(  1-\bar{\zeta
}z\right)  \Lambda^{1/4}}.
\]

We break $\mathbb{D}$ into regions beginning with $R_{0}=\mathbb{D}%
\cap\left\{  w:\left\vert w\right\vert \leq\gamma\right\}  .$ In $R_{0}$ we
have $\operatorname{Re}w\leq0$ and hence $\left\vert \gamma-w+\gamma
w\right\vert \geq\operatorname{Re}\left(  \gamma-w+\gamma w\right)  \geq
\gamma+O(\gamma^{2})\geq c\gamma.$ Hence for $z\in R_{0}$
\begin{align}
\gamma &  \lesssim\left\vert 1-\bar{\zeta}z\right\vert \lesssim1,\label{r0}\\
1+\left\vert \log\gamma\right\vert  &  \lesssim\left\vert \Lambda\right\vert
.\nonumber
\end{align}
For $n=1,2,...$ we set $R_{n}=\mathbb{D}\cap\left\{  w:2^{n-1}\gamma
\leq\left\vert w\right\vert \leq2^{n}\gamma\right\}  $ and denote by $n_{0}$
the largest $n$ for which $R_{n}\neq\phi;$ thus $n_{0}\sim\left\vert
\log\delta\right\vert .$ For $z\in R_{n},$ $1\leq n\leq n_{0}$, we have
\begin{align}
\left\vert 1-\bar{\zeta}z\right\vert  &  \sim2^{n}\gamma,\label{rn}\\
1+\left\vert \log2^{n}\gamma\right\vert  &  \lesssim\left\vert \Lambda
\right\vert .\nonumber
\end{align}
Taking the second of those estimates into account we have
\begin{align*}
\left\Vert G_{1}\right\Vert _{_{\mathcal{D}}}^{2}  &  \lesssim\int
_{\mathbb{D}}\left\vert \frac{\delta^{\theta}}{\left(  1-\bar{\zeta}z\right)
^{\theta+1}\Lambda^{3/4}}\right\vert ^{2}+\left\vert \frac{\delta^{\theta}%
}{\left(  1-\bar{\zeta}z\right)  ^{\theta+1}\Lambda^{7/4}}\right\vert ^{2}\\
&  \lesssim\delta^{2\theta}\int_{\mathbb{D}}\left\vert \frac{1}{\left(
1-\bar{\zeta}z\right)  ^{\theta+1}\Lambda^{3/4}}\right\vert ^{2}\\
&  \lesssim\delta^{2\theta}\sum_{n=0}^{n_{0}}\int_{R_{n}}\frac{1}{\left\vert
1-\bar{\zeta}z\right\vert ^{2\theta+2}\left\vert \Lambda\right\vert ^{3/2}}.
\end{align*}
Next we estimate each integral using (\ref{r0}) or (\ref{rn}) and the fact
that $\operatorname{Area}(R_{n})\lesssim2^{2n}\gamma^{2}\sim2^{2n}\delta^{2}.$
We continue with
\begin{align*}
\left\Vert G_{1}\right\Vert _{_{\mathcal{D}}}^{2}  &  \lesssim\delta
^{2\theta+2}\sum_{n=0}^{n_{0}}2^{2n}\frac{1}{\left(  2^{n}\delta\right)
^{2\theta+2}\left(  1+\left\vert \log2^{n}\gamma\right\vert \right)  ^{3/2}}\\
&  \lesssim\sum_{n=0}^{n_{0}}\frac{1}{2^{2n\theta}\left(  1+\left\vert
\log2^{n}\gamma\right\vert \right)  ^{3/2}}.\\
&  \lesssim\sum_{n=0}^{A}+\sum_{n=A}^{n_{0}}%
\end{align*}
where $A$ is the largest integer for which $2^{A}\sqrt{\gamma}<1.$ Thus
\begin{align}
\left\Vert G_{1}\right\Vert _{_{\mathcal{D}}}^{2}  &  \lesssim\sum_{n=0}%
^{A}+\sum_{n=A}^{n_{0}}\nonumber\\
&  \lesssim\frac{1}{\left\vert \log\gamma\right\vert ^{3/2}}\sum_{n=0}%
^{A}\frac{1}{2^{2n\theta}}+\sum_{n=A}^{n_{0}}\frac{1}{2^{2n\theta}}\nonumber\\
&  \lesssim\frac{1}{\left\vert \log\gamma\right\vert ^{3/2}}+\frac{1}{\left(
2^{A}\right)  ^{2\theta}}\lesssim\frac{1}{\left\vert \log\gamma\right\vert
^{3/2}}+\delta^{\theta}\nonumber\\
&  \lesssim\frac{1}{\left\vert \log\gamma\right\vert ^{3/2}}\sim\frac
{1}{\left\vert \log\delta\right\vert ^{3/2}}. \label{g1}%
\end{align}
Next%
\begin{align}
\left\Vert G_{2}\right\Vert _{_{\mathcal{D}}}^{2}  &  =\int_{\mathbb{D}%
}\left\vert G_{2}^{\prime}\right\vert ^{2}\lesssim\int_{\mathbb{D}}\frac
{1}{\left\vert 1-\bar{\zeta}z\right\vert ^{2}\left\vert \Lambda\right\vert
^{1/2}}\nonumber\\
&  \lesssim\sum_{n=0}^{n_{0}}\int_{R_{n}}\frac{1}{\left\vert 1-\bar{\zeta
}z\right\vert ^{2}\left\vert \Lambda\right\vert ^{1/2}}\nonumber\\
&  \lesssim\sum_{n=0}^{n_{0}}\left(  2^{n}\delta\right)  ^{2}\frac{1}{\left(
2^{n}\delta\right)  ^{2}\left(  1+\left\vert \log2^{n}\gamma\right\vert
\right)  ^{1/2}}\nonumber\\
&  \lesssim\sum_{n<c\left\vert \log\delta\right\vert }\frac{1}{\left(
1+\left\vert \log\gamma\right\vert -n\log2\right)  ^{1/2}}\nonumber\\
&  \lesssim\left\vert \log\gamma\right\vert ^{1/2}\sim\left\vert \log
\delta\right\vert ^{1/2}. \label{g2}%
\end{align}
Using (\ref{g1}) and (\ref{g2}) in (\ref{product}) and recalling that
$L(\zeta)\sim\left\vert \log\delta(z)\right\vert $ completes the proof of the lemma.
\end{proof}
\end{lemma}

We now have the upper estimates for (\ref{xderiv}) and (\ref{bump}). Next pick
and fix $\theta>0$ and $\zeta\in\mathcal{D}$. Recalling the duality of
$\mathcal{D}\odot\mathcal{D}$ and $\mathcal{X}$ we have
\[
\delta(\zeta)^{-1}\sim\left\vert \left.  \frac{d}{dz}G_{\zeta,\theta
}(z)\right\vert _{z=\zeta}\right\vert =\left\vert \left\langle G_{\zeta
,\theta},\bar{\partial}_{\zeta}k_{\zeta}\right\rangle \right\vert
\leq\left\Vert G_{\zeta,\theta}\right\Vert _{\mathcal{D}\odot\mathcal{D}%
}\left\Vert \bar{\partial}_{\zeta}k_{\zeta}\right\Vert _{\mathcal{X}}%
\]
which, given the upper estimates, forces the corresponding lower estimates.

We move to estimates for norms in $\mathcal{D}\odot\mathcal{D}$. First we
consider (\ref{tval}). By (\ref{square}) it is enough to show that $\left\Vert
k_{\zeta}^{1/2}\right\Vert _{\mathcal{D}}^{2}\lesssim\log(1+L(\zeta)).$ We
argue as in the previous lemma. Continuing the notation from that lemma we
have
\begin{align*}
\left\Vert k_{\zeta}^{1/2}\right\Vert _{\mathcal{D}}^{2}  &  \lesssim
\int_{\mathbb{D}}\left\vert k_{\zeta}^{-1/2}k_{\zeta}^{\prime}\right\vert
^{2}\sim\int_{\mathbb{D}}\frac{1}{\left\vert 1-\bar{\zeta}z\right\vert
^{2}\left\vert \log(1-\zeta z)\right\vert }\\
&  \lesssim\sum_{n=0}^{n_{0}}\left(  2^{n}\delta\right)  ^{2}\frac{1}{\left(
2^{n}\delta\right)  ^{2}\inf\left\{  \left\vert \log(1-\zeta z)\right\vert
:z\in R_{n}\right\}  }\\
&  \lesssim\sum_{n<c\left\vert \log\delta\right\vert }\frac{1}{\left(
1+\left\vert \log\gamma\right\vert -n\log2\right)  }\\
&  \lesssim\left\vert \log\left(  c_{1}+c_{2}\left\vert \log\delta\right\vert
\right)  \right\vert ^{1/2}%
\end{align*}
for some positive constant $c_{1},c.$ Recalling the relationship between
$\delta$ and $L$ completes the proof of the upper estimate for (\ref{tval}).

The upper estimate for (\ref{tsquare}) is an immediate consequence of
(\ref{square}) and (\ref{x}). The lower estimate is a consequence of the upper
estimate for (\ref{xval}), duality and the computation%
\[
L(\zeta)^{2}=\left\langle k_{\zeta}^{2},k_{\zeta}\right\rangle \leq\left\Vert
k_{\zeta}^{2}\right\Vert _{\mathcal{D}\odot\mathcal{D}}\left\Vert k_{\zeta
}\right\Vert _{\mathcal{X}}\sim\left\Vert k_{\zeta}^{2}\right\Vert
_{\mathcal{D}\odot\mathcal{D}}L(\zeta).
\]

For (\ref{tderiv}) one checks that
\[
\left\vert \frac{d}{dz}\left(  \bar{\partial}_{\zeta}k_{\zeta}\right)
\right\vert \sim\frac{1}{\delta}\left\vert \frac{1-\left\vert \zeta\right\vert
^{2}}{1-\bar{\zeta}z}\right\vert .
\]
The upper and lower estimates (\ref{tderiv}) now follow from those in
(\ref{bump}).

We now establish (\ref{H}) and hence also the lower estimate for
(\ref{tderiv}). If $\mu$ is a positive measure supported on $\left[
1/2,1\right]  $ which satisfies
\begin{equation}
\mu([x,1])\lesssim\frac{-1}{\log\left(  1-x\right)  }. \label{total}%
\end{equation}
then $\mu\in CM(\mathcal{D)}$ see, for instance, \cite{ARS}. In particular the
measure on $\left[  1/2,1\right]  $ given by
\[
d\mu(x)=\frac{1}{\left(  1-x\right)  \left[  \log\left(  1-x\right)  \right]
^{2}}dx
\]
is in $CM(\mathcal{D)}$. Hence by Theorem \ref{baylage} below we know that
$\left\Vert H\right\Vert _{\mathcal{X}}<\infty.$ Next we estimate $H(\zeta)$
for $\zeta\in(.9,1).$ Pick and fix $\zeta.$
\begin{align*}
H(\zeta)  &  =\int_{1/2}^{1}\log\left(  \frac{1}{1-\zeta x}\right)  \frac
{1}{\left(  1-x\right)  \left[  \log\left(  1-x\right)  \right]  ^{2}}dx\\
&  =\int_{1/2}^{\zeta}\cdot\cdot\cdot+\int_{\zeta}^{1}\cdot\cdot\cdot.
\end{align*}
Now note that
\begin{align*}
\frac{1-\zeta x}{1-x}  &  =O(1);\text{ }\frac{1}{2}\leq x\leq\zeta,\\
\frac{1-\zeta x}{1-\zeta}  &  =O(1);\text{ }\zeta\leq x\leq1,\text{ and}\\
\int_{1/2}^{1}d\mu &  <\infty.
\end{align*}
Hence we can continue our estimate of $H(\zeta)$ with
\begin{align*}
H(\zeta)  &  =\int_{1/2}^{\zeta}-\log(1-x)\frac{dx}{\left(  1-x\right)
\left[  \log\left(  1-x\right)  \right]  ^{2}}\\
&  \qquad\qquad-\log\left(  1-\zeta\right)  \int_{\zeta}^{1}\frac{dx}{\left(
1-x\right)  \left[  \log\left(  1-x\right)  \right]  ^{2}}+O(1)\\
&  =\int_{1/2}^{\zeta}\frac{-dx}{\left(  1-x\right)  \log\left(  1-x\right)
}-\log\left(  1-\zeta\right)  \int_{\zeta}^{1}\frac{dx}{\left(  1-x\right)
\left[  \log\left(  1-x\right)  \right]  ^{2}}+O(1)\\
&  =\log(\left\vert \log(1-\zeta)\right\vert )+O(1)-\log\left(  1-\zeta
\right)  \frac{-1}{\log(1-\zeta)}\\
&  =\log(\left\vert \log(1-\zeta)\right\vert )+O(1)
\end{align*}
as required.

To obtain the upper estimate in$\ ($\ref{ten}) note that $\left\Vert
z^{n}\right\Vert _{\mathcal{D}\odot\mathcal{D}}\leq\left\Vert z^{n}\right\Vert
_{\mathcal{D}}\left\Vert 1\right\Vert _{\mathcal{D}}\sim\sqrt{n}.$ To obtain
the upper estimate in$\ ($\ref{X}) we compute the Carleson measure norm of the
measure $d\mu=\left\vert f^{\prime}\right\vert ^{2}dxdy$ for the function
$f(z)=z^{n}.$ The measure only depends on $\left\vert z\right\vert $ hence the
norm is the maximum of the quantities
\[
A_{k}=\frac{1}{\left\Vert z^{k}\right\Vert _{\mathcal{D}}}\left(
\int\left\vert z^{k}\right\vert ^{2}d\mu\right)  ^{1/2},\text{ }k=0,1,2,...
\]
Doing the integration yields $A_{k}\sim n/\sqrt{\left(  k+1\right)  \left(
n+k+1\right)  }$ which has a maximum $n/\sqrt{n+1}$ $\sim\sqrt{n}$ at
$k=0.$These two upper estimates imply the two lower estimates because
$n\sim\left\langle z^{n},z^{n}\right\rangle \lesssim\left\Vert z^{n}%
\right\Vert _{\mathcal{D}\odot\mathcal{D}}\left\Vert z^{n}\right\Vert
_{\mathcal{X}}.$
\end{proof}

\begin{remark}
There is an alternative approach to the upper estimate in (\ref{tval}), the
growth estimates for $f\in\mathcal{X}$. If $f\in\mathcal{X}$ then for
$k=0,1,2,...$
\begin{align*}
\left\Vert f^{k+1}\right\Vert _{\mathcal{D}}^{2}  &  =\int\int\left\vert
\frac{d}{dz}f^{k+1}\right\vert ^{2}dA\\
&  =\int\int\left(  k+1\right)  ^{2}\left\vert f^{\prime}\right\vert
^{2}\left\vert f^{k}\right\vert ^{2}dA\\
&  \leq\left(  k+1\right)  ^{2}\left\Vert \left\vert f^{\prime}\right\vert
^{2}dA\right\Vert _{CM(\mathcal{D)}}\left\Vert f^{k}\right\Vert _{\mathcal{D}%
}^{2}\\
&  \leq\left(  k+1\right)  ^{2}\left\Vert f\right\Vert _{\mathcal{X}}%
^{2}\left\Vert f^{k}\right\Vert _{\mathcal{D}}^{2}.
\end{align*}
This computation gives both the starting and inductive step in showing
$\left\Vert f^{k}\right\Vert _{\mathcal{D}}=O(k!)\left\Vert f\right\Vert
_{\mathcal{X}}.$ Using those estimates we see that if $\varepsilon$ is small
we can sum the series for $g=\exp\left(  \varepsilon f\right)  $ and conclude
that $g\in\mathcal{D}$. The upper estimate in (\ref{tval}) follows from
applying (\ref{dvalue}) to the function $g\in$ $\mathcal{D}.$
\end{remark}

\section{Coefficients}

The norm of a function in $\mathcal{D}$ is unchanged if each Taylor
coefficient is replaced by its modulus$.$ This has consequences for the Taylor
coefficients of functions in $\mathcal{D}\odot\mathcal{D}$ and those in
$\mathcal{X}$.

\begin{theorem}
We have

\begin{enumerate}
\item If $a(z)=\sum a_{n}z^{n}\in\mathcal{D}\odot\mathcal{D}$ then there is a
$b(z)=\sum b_{n}z^{n}\in\mathcal{D}\odot\mathcal{D}$ with $\left\vert
a_{n}\right\vert \leq b_{n}$ and $\left\Vert b\right\Vert _{\mathcal{D}%
\odot\mathcal{D}}\leq C\left\Vert a\right\Vert _{\mathcal{D}\odot\mathcal{D}%
}.$

\item Suppose $c(z)=\sum c_{n}z^{n}\in\mathcal{X}$ with $c_{n}\geq0.$ Given
$\left\{  d_{n}\right\}  $ with $\left\vert d_{n}\right\vert \leq c_{n}$ then
$d(z)=\sum d_{n}z^{n}\in\mathcal{X}$ and $\left\Vert d\right\Vert
_{\mathcal{X}}\leq C\left\Vert c\right\Vert _{\mathcal{X}}.$
\end{enumerate}

\begin{proof}
The first statement is a direct consequence of the definitions and the comment
before the theorem. The second follows from the first and the duality
statement; item (2) in Theorem \ref{duality}.
\end{proof}
\end{theorem}

We will call a sequence of positive integers $n_{1}<n_{2}<...$ lacunary if
there is a $q>1$ so that $\forall k,n_{k+1}/n_{k}>q$ and we say that a
function $d(z)=\sum d_{n}z^{n}\ $has a lacunary poser series if $\left\{
n:d_{n}\neq0\right\}  $ is a lacunary sequence.

As we note in the proof, parts of the following theorem were first obtained by
Brown and Shields \cite{BS} building on earlier work by Taylor \cite{T}.

\begin{theorem}
We have

\begin{enumerate}
\item If $a(z)=\sum a_{n}z^{n}\in\mathcal{D}\odot\mathcal{D}$ then $\sum
\frac{\left\vert a_{n}\right\vert }{1+\log(n+1)}<\infty.$

\item Suppose $b(z)=\sum b_{n}z^{n};$

\begin{enumerate}
\item If $\left\vert b_{n}\right\vert \leq\frac{C}{\left(  n+1\right)  \left(
1+\log(n+1)\right)  }$ then $b(z)\in\mathcal{X}.$

\item If $\sum n\log\left(  n+1\right)  \left\vert b_{n}\right\vert
^{2}<\infty$ then $b(z)\in\mathcal{X}.$
\end{enumerate}

\item Suppose $d(z)$ has a lacunary power series; then the following are equivalent:

\begin{enumerate}
\item $d\in\mathcal{M}(\mathcal{D}),$

\item $d\in\mathcal{X},$

\item $d\in\mathcal{D},$

\item $d\in\mathcal{D}\odot\mathcal{D}$.
\end{enumerate}

\item Suppose $f(z)=\sum f_{n}z^{n}\in\mathcal{D}\odot\mathcal{D}$; then for
any lacunary set N,
\[
\sum_{n\in N}n\left\vert f_{n}\right\vert ^{2}<\infty.
\]

\end{enumerate}
\end{theorem}

\begin{proof}
By the first part of the previous theorem it suffices to prove the first
statement for a power series with positive coefficients, $a(z)=\sum a_{n}%
z^{n},$ $a_{n}\geq0.$ In that case we know $a(x)\geq0$ for $0\leq x<1. $
Select a positive measure $\mu$ supported on the interval $\left(  0,1\right)
$ with the property that $\mu((x,1))\sim(\log\left(  1-x\right)  )^{-1}.$ Such
a measure will be a Carleson measure for $\mathcal{D}$. Hence $a\rightarrow
\int ad\mu$ is a bounded linear functional on $\mathcal{D}\odot\mathcal{D}$
(Proposition \ref{prop} below). We also have
\[
\int r^{n}d\mu\geq\int\limits_{1-1/n}^{1}r^{n}d\mu\geq c\int\limits_{1-1/n}%
^{1}d\mu\geq\frac{c}{\left\vert \log\left(  1/n\right)  \right\vert }.
\]
Thus
\[
\left\Vert a\right\Vert \geq C\int ad\mu\geq C^{\prime}\sum\frac{a_{n}}%
{1+\log(n+1)}.
\]
Part (2a) follows from part (1) together with part (2) of the previous
theorem. The statement (2b) is a result of Brown and Shields. Although they
use different language they prove (2b) on page 299 of \cite{BS}.

We turn to (3). By the inclusions (\ref{dirinc}) we have (3a) $\Longrightarrow
$ (3b) $\Longrightarrow$ (3c) $\Longrightarrow$ (3d). Proposition 20 of
\cite{BS} is the statement that for lacunary series (3c) $\Longrightarrow$
(3a). To finish we show that (3d) $\Longrightarrow$ (3c). Suppose $d(z)$ is
given by a lacunary series and is in $\mathcal{D\odot D}$.\ We want to show
$d\in\mathcal{D}.$ Because $\mathcal{D}$ is a Hilbert space it suffices to
have good estimates of $\left\vert \left\langle d,h\right\rangle \right\vert $
for $h\in\mathcal{D}$, $\left\Vert h\right\Vert =1.$ If we replace $h$ by $j$
which has the same Taylor coefficients as $h$ for the indices for which
$d_{n}\neq0$ and has its other coefficients $0$ then we have both
$\left\langle d,h\right\rangle =\left\langle d,j\right\rangle $ and
$\left\Vert j\right\Vert \leq$ $\left\Vert h\right\Vert .\ $Hence it suffices
to estimate $\left\vert \left\langle d,j\right\rangle \right\vert .$ Using the
fact that (3c) $\Longrightarrow$ (3a), the inclusions (\ref{dirinc}), and item
(2) of Theorem \ref{duality} we have%
\[
j\in\mathcal{M}(\mathcal{D})\subset\mathcal{X=}\left(  \mathcal{D\odot
D}\right)  ^{\ast}.
\]
Hence $d$, which we assumed was in $\mathcal{D\odot D}$ pairs with $j$ with
the appropriate estimates.

To prove (4) note that by (3) given any sequence $\left\{  g_{n}\right\}
_{n\in N}$ with $\sum_{n\in N}n\left\vert g_{n}\right\vert ^{2}=1$ the
function $g=\sum g_{n}z^{n}$ is in $\mathcal{X}$ with uniformly bounded norm.
Pairing $g$ with $f$ and taking the supremum over $g$ gives the conclusion.
\end{proof}

Part (2a) of the theorem rests on the fact that Carleson measures supported on
the interval (0,1) are easy to characterize. Part (2b) of the theorem rests on
the fact that it is easy to characterize the measures $\mu$ for which the
natural densely defined inclusion of $\mathcal{D}$ into $L^{2}(\mu)$ extends
to a map in the Hilbert Schmidt class; see the second part of Theorem
\ref{sch} below.

With one exception these results are analogous to Hardy space results. Part
(1) is the analog of Hardy's inequality which states that if $a\in H^{1}$ then
$\sum\left\vert a_{n}\right\vert /\left(  n+1\right)  <\infty.$ The duality of
$H^{1}$ and $BMO$ then gives an analog of (2a). Statement (2b) is the analog
of the fact that the \textit{Hilbert} \textit{space} $\mathcal{D}$ is
contained in $BMO$, a standard result which can be given a simple proof by
adapting the proof of (2b) in \cite{BS}. Statement (3) is the analog of the
basic Littlewood-Paley result for the Hardy spaces: if $f(z)$ has a lacunary
power series and is in one of the spaces $H^{p},$ $p>0,$ or $BMO$ then it is
in all of them. This is where there is a small exception to the general
analogy. It is straightforward that having $f$ in $H^{2}$ with a lacunary
power series does not force $f\in H^{\infty}$ and it is equally
straightforward that if such an $f$ is in $\mathcal{D}$ then it is bounded.
The final statement is the analog of Paley's theorem that if $f\in H^{1}$ and
$N$ is a lacunary set then $\sum_{n\in N}\left\vert f_{n}\right\vert
^{2}<\infty.$

\section{Carleson Measures and Interpolation}

\subsection{Carleson Measures}

We will say $\mu$ is a Carleson measure for $\mathcal{D}\odot\mathcal{D}$,
$\mu\in CM(\mathcal{D}\odot\mathcal{D)}$, if there is a $C>0$ so that for all
$f$ in $\mathcal{D}\odot\mathcal{D}$
\[
\int_{\mathbb{D}}\left\vert f\right\vert d\mu\leq C\left\Vert f\right\Vert
_{\mathcal{D}\odot\mathcal{D}}.
\]

\begin{proposition}
\label{prop}$CM(\mathcal{D}\odot\mathcal{D)}=CM(\mathcal{D)}.$

\begin{proof}
This follows immediately from the definitions, (\ref{square}), and the
Cauchy-Schwarz inequality.
\end{proof}
\end{proposition}

One can also ask for which \textit{complex} measures $\int_{\mathbb{D}}fd\mu$
will be bounded; that is, when are there estimates%

\begin{align}
\left\vert \int_{\mathbb{D}}fd\mu\right\vert  &  \leq C\left\Vert f\right\Vert
_{\mathcal{D}\odot\mathcal{D}}\text{ }\forall f\in\mathcal{D}\odot
\mathcal{D}\text{,}\label{a}\\
\left\vert \int_{\mathbb{D}}f^{2}d\mu\right\vert  &  \leq C\left\Vert
f\right\Vert _{\mathcal{D}}^{2}\text{ }\forall f\in\mathcal{D}\text{?}
\label{b}%
\end{align}
The answer is the same in both cases. Given a finite complex measure $\mu$
define its Dirichlet projection $P_{\mathcal{D}}\bar{\mu}(w)$ by%
\[
P_{\mathcal{D}}\bar{\mu}(w)=\int_{\mathbb{D}}\log\left(  \frac{1}{1-w\bar{z}%
}\right)  d\bar{\mu}(z).
\]
(Here $\bar{\mu}$ is the complex conjugate of the measure $\mu.)$

\begin{theorem}
\label{baylage}Given a finite complex measure $\mu$ on the disk, estimate
(\ref{a}), or equivalently (\ref{b}), holds if and only if $P_{\mathcal{D}%
}\bar{\mu}\in\mathcal{X}$.

\begin{proof}
To obtain the first statement compute with monomials to check that
$\int_{\mathbb{D}}fd\mu=\left\langle f,P_{B}\mu\right\rangle $ and then invoke
part (2) of Theorem \ref{duality}. If $P_{\mathcal{D}}\mu\in\mathcal{X}$ then,
evaluating (\ref{a}) on the function $f^{2}$ and taking note of (\ref{square})
we see that (\ref{b}) holds. Finally we note that if $\mu$ is given and
(\ref{b}) holds then so does (\ref{a}). The reason is that, again, noting
(\ref{pol}), if $g\in\mathcal{D}\odot\mathcal{D}$ then $g$ can be written as
$g=\sum h_{j}^{2}$ with $h_{j}\in\mathcal{D}$ and $\sum\left\Vert
h_{j}\right\Vert _{\mathcal{D}}^{2}\leq C\left\Vert g\right\Vert
_{\mathcal{D}\odot\mathcal{D}}.$
\end{proof}
\end{theorem}

\begin{corollary}
If $\mu\in CM(\mathcal{D)}$ then
\[
B_{\mu}\left(  w\right)  =\int_{\mathbb{D}}\left(  \pi+\arg\left(  1-w\bar
{z}\right)  \right)  \text{ }d\mu(z).
\]
is the real part of a function in $\mathcal{X}$. (We are using the branch of
$\arg$ for which $\left\vert \arg(\zeta)\right\vert \leq\pi.)$\ 

\begin{proof}
By bringing absolute values inside the integral we see that $\mu$ satisfies
(\ref{a}). Hence $P_{\mathcal{D}}\bar{\mu}\in\mathcal{X}$. Using the fact that
$\mu$ is real we find $B_{\mu}=\operatorname{Re}\left(  -iP_{\mathcal{D}}%
\bar{\mu}+C\right)  $ for come constant $C,$ as required.
\end{proof}
\end{corollary}

This corollary is the analog of the fact that the bayalage of a Carleson
measure for the Hardy space, a function obtained from the measure by
integrating against a certain positive kernel, is the real part of a function
in $BMO,$ \cite[Ch 4 Th 1.6]{G}.

\subsection{Interpolating Sequences}

Let $Z=\left\{  z_{i}\right\}  $ be a sequence of points in the open disk. The
associated measure $\mu_{Z}$ is defined by
\[
\mu_{Z}=\sum_{j=1}^{\infty}\frac{1}{L(z_{j})}\delta_{z_{j}}\text{.}%
\]
Let $R$ be the restriction map which takes a holomorphic function $f$ to its
sequence of values on $Z$, $\left\{  f(z_{i})\right\}  .$ The sequence $Z$ is
said to be an interpolating sequence for $\mathcal{D}$ if $R$ is a bounded map
of $\mathcal{D}$ into and onto $l^{2}(\mu_{Z})\ $and an interpolating sequence
for $\mathcal{D}\odot\mathcal{D}$ if $R$ maps $\mathcal{D}\odot\mathcal{D}$
boundedly into and onto $l^{1}(\mu_{Z}).$ It is automatic that $R$ maps
$\mathcal{M}(\mathcal{D})$ into $l^{\infty}(\mu_{Z});$ if the map is
surjective we say $Z$ is an interpolating sequence for $\mathcal{M}%
(\mathcal{D}).$ Walking through the definitions shows that if $R$ is bounded
on either $\mathcal{D}$ or $\mathcal{D}\odot\mathcal{D}$ then we must have
that
\begin{equation}
\mu_{Z}\text{ is a Carleson measure.} \label{CM}%
\end{equation}
(We noted earlier that $CM(\mathcal{D}\odot\mathcal{D)}=CM(\mathcal{D})$ so it
is not necessary to specify further.) Also, in order for interpolation to be
possible the points of $Z$ must maintain an appropriate distance from each
other. We will say $Z$ is separated if there is a $C>0$ so that for all
$i,j,i\neq j$%
\begin{equation}
\beta(0,z_{i})\leq C\beta(z_{i},z_{j}). \label{Sep}%
\end{equation}

\begin{theorem}
The following are equivalent for a sequence Z:

\begin{enumerate}
\item Z satisfies (\ref{CM}) and (\ref{Sep}),

\item Z is an interpolating sequence for $\mathcal{D},$

\item Z is an interpolating sequence for $\mathcal{M}(\mathcal{D}),.and$

\item Z is an interpolating sequence for $\mathcal{D}\odot\mathcal{D}.$
\end{enumerate}

\begin{proof}
The equivalence of the first three statements was shown in manuscripts
circulated by Marshall-Sundberg \cite{MS} and by Bishop \cite{Bi}. The first
published proof is due to B\"{o}e \cite{Bo}. Our contribution is the
equivalence of the last statement.

To show this we first note that if (\ref{CM}) holds then $R$ maps
$\mathcal{D}\odot\mathcal{D}$ boundedly into $l^{1}(\mu_{Z}).$ To see that the
map is onto, suppose $\alpha=\left\{  \alpha_{i}\right\}  $ $\in$ $l^{1}%
(\mu_{Z})$ and we wish to find $f$ $\in\mathcal{D}\odot\mathcal{D}$ with
$f(z_{i})=\alpha_{i},$ $i=1,2,...$ Consider sequences $\beta$ and $\gamma$
defined, for $i=1,2,...$ by
\[
\beta_{i}=\left\vert \alpha_{i}\right\vert ^{1/2},\text{ }\gamma
_{i}=\left\vert \alpha_{i}\right\vert ^{1/2}\frac{\alpha_{i}}{\left\vert
\alpha_{i}\right\vert }%
\]
and note that $\beta$ and $\gamma$ are in $l^{2}(\mu_{Z}).$ If (\ref{CM}) and
(\ref{Sep}) are satisfied then by the second statement Z is an interpolating
sequence for $\mathcal{D}$. Hence we can find $b$ and $g$ in $\mathcal{D}$ so
that for all $i,$ $b(z_{i})=\beta_{i}$ $,g(z_{i})=\gamma_{i}.$ The function
$f=bg$ is the function in $\mathcal{D}\odot\mathcal{D}$ that we require.

In the other direction, suppose Z is an interpolating sequence for
$\mathcal{D}\odot\mathcal{D}.$ We have noted that if $R$ is into then
(\ref{CM}) holds. To finish we show that having $R$ be onto with the natural
norm estimates forces Z to satisfy (\ref{Sep}). Pick $x,y\in Z,x\neq y.$
Suppose $\beta(0,x)\leq\beta(0,y)$ Because $Z$ is an interpolating sequence we
can find $f\in\mathcal{D}\odot\mathcal{D}$ with $f(x)=0,f(y)=L(y)$ and
$\left\Vert f\right\Vert _{\mathcal{D}\odot\mathcal{D}}=O(1).$ We now use
(\ref{xdiff}).%
\begin{align*}
\beta(0,y)  &  \sim L(y)=|f(x)-f(y)|\\
&  \lesssim\left\Vert f\right\Vert _{\mathcal{D}\odot\mathcal{D}}\left\Vert
k_{x}-k_{y}\right\Vert _{\mathcal{X}}\\
&  \lesssim\left\Vert f\right\Vert _{\mathcal{D}\odot\mathcal{D}}%
\beta(x,y)^{1/2}\left(  L(x)^{1/2}+L(y)^{1/2}\right) \\
&  \lesssim\beta(x,y)^{1/2}L(y)^{1/2}\\
&  \lesssim\beta(x,y)^{1/2}\beta(0,y)^{1/2}.
\end{align*}
Hence $\beta(0,y)\leq C\beta(x,y)$, as required.
\end{proof}
\end{theorem}

\section{Hankel Type Matrices and Schatten Classes\label{matrices}}

\subsection{The Hardy Space}

A Hankel form on the Hardy space is a bilinear form generated by a holomorphic
symbol function $b$ through the formula%
\[
H_{b}^{\text{Hardy}}\left(  f,g\right)  =\left\langle fg,b\right\rangle
_{\text{Hardy}}.
\]
If $b(z)=\sum_{0}^{\infty}b(n)z^{n}$ then the matrix representation of the
form with respect to the standard orthonormal basis of monomials is $\left(
\overline{b(i+j)}\right)  _{i,j=0}^{\infty}.$ Matrices of this form, the
$\left(  i,j\right)  $ entry is a function of $i+j,$ are called Hankel
matrices. Straightforward functional analytic considerations show that the
space of $b$ for which the form is bounded is exactly the space $\left(
H^{2}\odot H^{2}\right)  ^{\ast}.$ (Such an argument is given in detail in,
for instance, \cite{ARSW}). When this is combined with Fefferman's
identification of $\left(  H^{1}\right)  ^{\ast}$ with $BMO$ we obtain the
first statement of the next theorem. That statement is an endpoint of a scale
of statements relating the size of the $H_{b}^{\text{Hardy}}$ to the
smoothness of the function $b;$ for $0<p<\infty$ the form $H_{b}%
^{\text{Hardy}}$ is in the Schatten class $\mathcal{S}_{p}$ if and only if $b$
is in the diagonal Besov space $B_{p}.$ These ideas are presented
systematically in \cite{P} and \cite{N}, here we just recall a few specifics.
The class $\mathcal{S}_{2}$ is the Hilbert-Schmidt class; it consists of
bilinear forms with the property that their matrix entries with respect to
some, and hence every, orthonormal basis are square summable. The class
$\mathcal{S}_{1}$ is the trace class; it consists of bilinear forms $K$ which
can be written as $K=\sum\alpha_{i}R_{i}$ where the $R_{i}$ are bilinear forms
of norm one and rank one and the sequence of scalars $\left\{  \alpha
_{i}\right\}  $ is absolutely summable. The Besov space $B_{2}$ coincides with
the Dirichlet space $\mathcal{D}$. The Besov space $B_{1}$ is defined by
condition (\ref{besov}) below. We have:

\begin{theorem}
\label{Hardy}

\begin{enumerate}
\item $H_{b}^{\text{Hardy}}$ is bounded if and only if $b\in BMO.$

\item $H_{b}^{\text{Hardy}}$ is in the Hilbert-Schmidt class if and only if
$b\in B_{2},$ i.e. if and only if
\[
\int_{\mathbb{D}}\left\vert b^{\prime}(z)\right\vert ^{2}dxdy<\infty.
\]

\item $H_{b}^{\text{Hardy}}$ is in the trace class if and only if $b\in
B_{1},$ i.e. if and only if
\begin{equation}
\int_{\mathbb{D}}\left\vert b^{\prime\prime}(z)\right\vert dxdy<\infty.
\label{besov}%
\end{equation}

\end{enumerate}
\end{theorem}

More generally matrices of the form%
\begin{equation}
A=\left(  a_{ij}\right)  =\left(  \left(  i+1\right)  ^{\alpha}(j+1)^{\beta
}\left(  i+j+1\right)  ^{\gamma}\overline{b(i+j)}\right)  ; \label{A}%
\end{equation}
with
\begin{equation}
\mathit{\min}\left\{  \alpha,\beta\right\}  >\max\left\{  \frac{-1}{2}%
,\frac{-1}{p}\right\}  \label{condition}%
\end{equation}
correspond to forms in $\mathcal{S}_{p}$ if and only if $b(z)$ has a certain
fractional order derivative in $B_{p};$ however that fails if (\ref{condition}%
) fails \cite[Ch 6 Thm 8.9]{P}.

\subsection{The Dirichlet Space}

By a Hankel form on the Dirichlet space we mean a form generated by a
holomorphic symbol function $b$ through the formula
\[
H_{b}^{\text{Dirichlet}}\left(  f,g\right)  =\left\langle fg,b\right\rangle
_{\text{Dirichlet}},
\]
or, more compactly, $H_{b}\left(  f,g\right)  =\left\langle fg,b\right\rangle
. $

It is convenient to restrict $H_{b}$ to the subspace of $\mathcal{D}$ of
functions which vanish at the origin and we do that for the rest of the
section. With that restriction there is no loss in assuming $b(0)=0;$ thus
$b(z)=\sum_{1}^{\infty}b_{n}z^{n}.$ The matrix representation of $H_{b}$ with
respect to the orthonormal basis of monomials $\left\{  n^{-1/2}z^{n}\right\}
_{1}^{\infty}$ is
\begin{equation}
B=\left(  \beta_{ij}\right)  =\left(  \frac{i+j+1}{\sqrt{i+1}\sqrt{j+1}%
}\overline{\hat{b}(i+j)}\right)  _{i,j=1}^{=}. \label{form}%
\end{equation}
This corresponds to $\alpha=\beta=-1/2$ in (\ref{A}), outside the range
(\ref{condition}). The form considered in \cite{CM}, \cite{To} and \cite{RW}
corresponds to $\alpha=-1/2$, $\beta=1/2$, also outside that range.

The boundedness criteria for the forms (\ref{form}) is known and, as expected,
the Hilbert-Schmidt criterion is straightforward to obtain. One approach to
the proof of the third statement in Theorem \ref{Hardy} is through the use of
decomposition theorems. When that approach is used to study trace class
membership for $H_{b}^{\text{Hardy}}$ one can obtain a necessary condition for
membership and a sufficient condition, and the two conditions obtained are the
same. However using a similar approach to study $H_{b}^{\text{Dirichlet}}$
produces two different conditions. We record those results in the following
theorem. They, together with their straightforward consequences by
interpolation, are the state of our knowledge.

\begin{theorem}
\label{sch}

\begin{enumerate}
\item $H_{b}$ is bounded if and only if $b\in\mathcal{X}.$

\item $H_{b}$ is in the Hilbert-Schmidt class if and only if
\[
\sum_{1}^{\infty}n\log n\left\vert b_{n}\right\vert ^{2}<\infty,
\]
equivalently if and only if%
\begin{equation}
\int_{\mathbb{D}}\left\vert b^{\prime}(z)\right\vert ^{2}\log\left(  \frac
{1}{1-\left\vert z\right\vert ^{2}}\right)  dxdy<\infty. \label{int}%
\end{equation}

\item \qquad

\begin{enumerate}
\item If $H_{b}$ is in the trace class then $b\in B_{1},$ i.e.
\begin{equation}
\int_{\mathbb{D}}\left\vert b^{\prime\prime}(z)\right\vert dxdy<\infty.
\label{trace}%
\end{equation}

\item If
\begin{equation}
\int_{\mathbb{D}}\left\vert b^{\prime\prime}(z)\right\vert \sqrt{\log\left(
\frac{1}{1-\left\vert z\right\vert ^{2}}\right)  }dxdy<\infty. \label{log}%
\end{equation}
then $H_{b}$ is in the trace class.

\item Neither of the two previous implications can be reversed.

\item In fact there is no function $\rho(r)$ increasing continuously to
$\infty$ on $\left(  0,1\right)  $ with the property that knowing $H_{b}$ is
in the trace class insures
\[
\int_{\mathbb{D}}\left\vert b^{\prime\prime}(z)\right\vert \rho(\left\vert
z\right\vert )dxdy<\infty.
\]

\end{enumerate}
\end{enumerate}
\end{theorem}

\begin{proof}
The first statement is in \cite{ARSW}.

For the second recall that $H_{b}$ is in $\mathcal{S}_{2}$ if and only if it
has square summable matrix entries. We start with the matrix (\ref{form}) and
compute%
\begin{align*}
\sum_{i,j}\left\vert \frac{i+j+1}{\sqrt{i+1}\sqrt{j+1}}b_{i+j}\right\vert
^{2}  &  =\sum_{k}\left(  \sum_{m=0}^{k}\frac{\left(  k+1\right)  ^{2}%
}{\left(  m+1\right)  \left(  k-m-1\right)  }\right)  \left\vert
b_{k}\right\vert ^{2}\\
&  =\sum_{k}\frac{\left(  k+1\right)  ^{2}}{k}\left(  \sum_{m=0}^{k}\frac
{1}{m+1}+\frac{1}{k-m-1}\right)  \left\vert b_{k}\right\vert ^{2}\\
&  =\sum_{k}\left(  2k\log k\right)  \left(  1+O(1)\right)  \left\vert
b_{k}\right\vert ^{2}%
\end{align*}
which is equivalent to the desired condition. The integral condition follows
from this together with the facts that the integral in (\ref{int}) equals%
\[
\sum n^{2}\left\vert b_{n}\right\vert ^{2}\int_{0}^{1}r^{2n-1}\log\left(
\frac{1}{1-r^{2}}\right)  dr
\]
and
\[
\int_{0}^{1}r^{2n-1}\log\left(  \frac{1}{1-r^{2}}\right)  dr\sim\frac{\log
n}{n}.
\]

The arguments for the first two parts of (3) are adaptations of arguments used
to obtain analogous statements for Hankel forms on the Hardy and Bergman
spaces; we only present the broad strokes here. One can find such arguments
presented in full in \cite{R} or \cite{Z}.

Pick a small parameter $\varepsilon>0.$ Pick a set of points $\Omega=\left\{
\omega_{i}\right\}  \subset\mathbb{D}$ that are hyperbolically separated;
$\beta\left(  \omega_{i},\omega_{j}\right)  >\varepsilon$ if $i\neq j$; and
also so that $\Omega$ is relatively thick,%
\[
\forall z\in\mathbb{D}\text{ }\inf\left\{  \beta(z,\omega_{i}):\omega_{i}%
\in\Omega\right\}  <100\varepsilon.
\]
If $\varepsilon$ is sufficiently small then the set of functions
\[
\left\{  h_{i}\right\}  =\left\{  \frac{\delta\left(  \omega_{i}\right)
}{1-\bar{\omega}_{i}z}:\omega_{i}\in\Omega\right\}  \subset\mathcal{D}%
\]
is the image of an orthonormal basis of a Hilbert space under a bounded linear
map. This insures that if $K$ is a trace class bilinear form then $\sum
_{i}\left\vert K\left(  h_{i},h_{i}\right)  \right\vert <\infty.$ Hence if
$H_{b}$ is trace class then
\[
\sum\left\vert \left\langle \frac{\delta\left(  \omega_{i}\right)  ^{2}%
}{\left(  1-\bar{\omega}_{i}z\right)  ^{2}},b\right\rangle \right\vert
<\infty.
\]
It is essentially true that $\left\langle \left(  1-\bar{\omega}_{i}z\right)
^{-2},b\right\rangle =b^{\prime\prime}(\omega_{i}).$ Hence
\[
\sum\left\vert b^{\prime\prime}(\omega_{i})\right\vert \delta\left(
\omega_{i}\right)  ^{2}<\infty.
\]
This sum is approximately a Riemann sum for the integral (\ref{trace}). Using
basic modulus of continuity estimates and the flexibility available in the
construction of $\Omega$ we can construct a similar sum which is a majorant
for the integral. That establishes (a).

To see that this implication cannot be reversed we consider the functions
$\bar{\partial}_{\zeta}k_{\zeta}\left(  \cdot\right)  $ and the associated
Hankel forms $H_{\zeta}$ given by $H_{\zeta}\left(  f,g\right)  =\left\langle
fg,\bar{\partial}_{\zeta}k_{\zeta}\right\rangle .$ We have
\[
H_{\zeta}\left(  f,g\right)  =\left(  fg\right)  ^{\prime}\left(
\zeta\right)  =f^{\prime}\left(  \zeta\right)  g\left(  \zeta\right)
+g^{\prime}\left(  \zeta\right)  f\left(  \zeta\right)  .
\]
Thus the $H_{\zeta}$ are all rank two forms and hence their trace class norms
are uniformly comparable to their operator norms. By the first part of this
theorem and the estimate (\ref{xderiv}) we conclude that%
\begin{equation}
\left\Vert H_{\zeta}\right\Vert _{\text{trace class}}\sim L(\zeta)^{1/2}%
\delta(\zeta)^{-1}. \label{tc}%
\end{equation}
On the other hand standard estimates show $\left\Vert \bar{\partial}_{\zeta
}k_{\zeta}\right\Vert _{B_{1}}\sim\delta(\zeta)^{-1}.$ Thus
\[
\left\Vert H_{\zeta}\right\Vert _{\text{trace class}}\nsim\left\Vert
\bar{\partial}_{\zeta}k_{\zeta}\right\Vert _{B_{1}}.
\]

We now go to the third part. Set $\tilde{H}_{\zeta}=L(\zeta)^{-1/2}%
\delta(\zeta)H_{\zeta}.$ By the discussion in the preceding paragraph these
are forms with uniformly bounded trace class norms. Hence any linear
combination of them with absolutely summable coefficients is also in
$\mathcal{S}_{1}.$ In particular, if we select a set $\Omega$ meeting the
conditions stated earlier and let $\{\alpha(\omega_{i})\}$ be a function
defined on $\Omega$ with the property that
\[
\sum\left\vert \alpha(\omega_{i})\right\vert L\left(  \omega_{i}\right)
^{1/2}<\infty.
\]
then the form $K,$
\begin{align*}
K  &  =\sum\alpha(\omega_{i})\delta(\omega_{i})H_{\omega_{i}}\\
&  =\sum\alpha(\omega_{i})L\left(  \omega_{i}\right)  ^{1/2}\tilde{H}%
_{\omega_{i}},
\end{align*}
is in $\mathcal{S}_{1}$ with norm dominated by $\sum\left\vert \alpha
(\omega_{i})\right\vert L\left(  \omega_{i}\right)  ^{1/2}.$

The symbol function $b$ of $K$ is
\[
b(z)=\sum\alpha(\omega_{i})\delta(\omega_{i})\left.  \!\bar{\partial}_{\zeta
}k_{\zeta}\right\vert _{\zeta=\omega_{i}}.
\]
To finish this part we invoke the following decomposition result which is a
straightforward variation of the results of \cite[Sect 4.5]{Z}.

\begin{lemma}
Suppose $\varepsilon$ is sufficiently small. Given a function $b$ which
satisfies (\ref{log}) one can find scalars $\left\{  \alpha(\omega
_{i})\right\}  $ so that
\begin{align*}
b(z)  &  =\sum\alpha(\omega_{i})\delta(\omega_{i})\left.  \!\bar{\partial
}_{\zeta}k_{\zeta}\right\vert _{\zeta=\omega_{i}}(z)\text{ and}\\
\sum\left\vert \alpha(\omega_{i})\right\vert L\left(  \omega_{i}\right)
^{1/2}  &  \lesssim\int_{\mathbb{D}}\left\vert b^{\prime\prime}(z)\right\vert
\sqrt{\log\left(  \frac{1}{1-\left\vert z\right\vert ^{2}}\right)  }dxdy.
\end{align*}

\end{lemma}

The fact that the implication in part (c) cannot be reversed is a special case
of statement (d). That statement is based on an observation of Bonami and
Bruna \cite[Thm 8]{BB}. Suppose $b(z)$ is given by the lacunary series
$b(z)=\sum c_{k}z^{3^{k}}.$ If $\sum3^{k}\left\vert c_{k}\right\vert
\,<\infty$ then the matrix entries of $H_{b}$ are absolutely summable and
hence $H_{b}$ is in the trace class. On the other hand given $\rho$ as
described it is straightforward to select the $\left\{  c_{k}\right\}  $ so
that the summability condition is met but $\int\left\vert b^{\prime\prime
}\right\vert \rho$\thinspace$dxdy=\infty.$\textbf{\ }
\end{proof}

\section{Questions}

Very little is known about the spaces $\mathcal{X}$ and $\mathcal{D}%
\odot\mathcal{D}$ or the functions in them. Here we mention some questions
that seem natural.

We lack a satisfactory intrinsic characterization of the functions in these
spaces. We noted that (\ref{bmo}) is equivalent to the traditional definition
of $BMO$. That definition is based on characterization of $BMO$ functions
using certain measures of oscillation. It would be interesting to have a
characterization of $\mathcal{X}$ based on local oscillation. It would also be
interesting to have a representation of functions in $\mathcal{D}%
\odot\mathcal{D}$ in terms of simple building blocks, analogous to the atomic
decomposition of $H^{1}.$

There are real and complex interpolation scales with the spaces $\mathcal{X}$
and $\mathcal{D}\odot\mathcal{D}$ as endpoints. The duality statements in
Theorem 1 and basic facts from interpolation theory insure that the midpoint
of those scales will be $\mathcal{D}$. However beyond that very special case
the authors do not even have attractive conjectures for the description of the
interpolation spaces.

For which $b$ is $H_{b}\in\mathcal{S}_{p}$? Answers to similar questions have
often involved Besov spaces. However the results in the previous theorem,
particularly for $p<2,$ suggest that may not be the case here.

\section{Acknowledgement}

This work was begun while the second, third, and fourth author were visiting
the Fields Institute. They thank the institute for its hospitality and its
excellent working conditions.

\end{document}